\documentclass[11pt]{article}
\usepackage[margin=1in]{geometry}  
\usepackage{graphicx}              
\usepackage{amsmath, amssymb, amsthm, epsfig}               
\usepackage{enumerate}
\usepackage{amsfonts}              
\usepackage{amsthm}                
\usepackage{amsthm}
\usepackage{enumerate}
\theoremstyle{plain}
\usepackage{color}
\usepackage{hyperref}
\hypersetup{colorlinks,citecolor=blue}
\usepackage{cite}
\newtheorem{corollary}{Corollary}[section]

\newtheorem{lemma}[corollary]{Lemma}
\newtheorem{prop}[corollary]{Proposition}

\newtheorem{thm}[corollary]{Theorem}

\newfont{\sBlackboard}{msbm10 scaled 1200}

\newcommand{\mylabel}[1]{\label{#1}
    \ifx\undefined\stillediting
    \else \fbox{$#1$}\fi }
\newcommand{\BE}{\begin{equation}}

\newcommand{\EEQ}{\end{equation}}
\newcommand{\rfb}[1]{\mbox{\rm
        (\ref{#1})}\ifx\undefined\stillediting\else:\fbox{$#1$}\fi}

\newfont{\Blackboard}{msbm10 scaled 1200}

\newfont{\roma}{cmr10 scaled 1200}

\newcommand{\bb}{\begin{equation}}
\newcommand{\bbb}{\end{equation}}

\newcommand{\mm}    {{\hbox{\hskip 0.5pt}}}

\newcommand{\bluff} {{\hbox{\raise 15pt \hbox{\mm}}}}
scaled\magstep2

\makeatletter
\def\section{\@startsection {section}{1}{\z@}{-3.5ex plus -1ex minus
        -.2ex}{2.3ex plus .2ex}{\large\bf}}
\makeatother
\numberwithin{equation}{section}

\begin{document}
\title{Existence of solutions for  nonlinear  Dirac equations in the Bopp-Podolsky electrodynamics}
\author {Hlel Missaoui\footnote{
hlel.missaoui@fsm.rnu.tn;\ hlelmissaoui55@gmail.com}\\
Mathematics Department, Faculty of Sciences, University of Monastir,\\ 5019 Monastir, Tunisia}
\maketitle


\begin{abstract}
In this paper, we study the following nonlinear  Dirac-Bopp-Podolsky system 
\begin{equation*}
\left\lbrace
\begin{array}{rll}
  \displaystyle{ -i\sum_{k=1}^{3}\alpha_{k}\partial_{k}u+[V(x)+q]\beta u+wu-\phi u}&=f(x,u),  \ \ &\text{in}\ \mathbb{R}^3,  \\
   \ & \ & \ \\
   -\triangle\phi+a^2\triangle^2 \phi&=4\pi \vert u\vert^2,\ \ &  \text{in}\ \mathbb{R}^3, 
\end{array}
\right.
\end{equation*}
where  $a,q>0,w\in \mathbb{R}$, $V(x)$ is a potential function, and $f(x, u)$ is the interaction term (nonlinearity). First, we give a physical motivation for this new kind of system. Second, under suitable assumptions on $f$ and $V$, and by  means of minimax techniques involving Cerami sequences, we prove the existence of at least one pair of solutions $(u,\phi_u)$.
\end{abstract}



\section{Introduction}
The Dirac equation, proposed by British physicist Paul Dirac in $1928$ (see \cite{PD}), is a relativistic wave equation that describes the behavior of particles with spin-$\frac{1}{2}$, such as electrons, in a relativistic quantum mechanical framework. The equation can be written as:
$$i\hbar\frac{\partial \psi}{\partial t}=ic\hbar \sum_{k=1}^{3}\alpha_k\partial_{k}\psi-mc^2\beta \psi,$$
where $\hbar$ is the reduced Planck constant, $\psi$ represents the wave function of the particle, $c$ is the speed of light, $m$ is the mass of the particle, and $\alpha_1,\alpha_2,\alpha_3$ and $\beta$ are the $4\times 4$ Pauli-Dirac matrices 
$$\beta=\begin{pmatrix}
I & 0\\
0 & -I
\end{pmatrix},\ \alpha_k=\begin{pmatrix}
    0 & \sigma_k\\
    \sigma_k & 0
\end{pmatrix},\ k=1,2,3,$$
with 
$$\sigma_1=\begin{pmatrix}
    0 & 1\\
    1 & 0
\end{pmatrix},\ \sigma_2=\begin{pmatrix}
    0 & -i \\
    i & 0
\end{pmatrix}, \ \text{and}\ \sigma_3=\begin{pmatrix}
    1 & 0\\
    0 & -1
\end{pmatrix}.$$
These matrices are constructed in a way that ensures that they are both Hermitian and satisfy the following anticommutation relations
$$\lbrace\alpha_k;\alpha_\ell\rbrace=2\delta_{k\ell},\ \lbrace\alpha_k;\beta\rbrace=\lbrace\alpha_k;\alpha_0\rbrace=\lbrace\beta;\alpha_0\rbrace=0,\ \text{and}\ \beta^2=1,$$
where $\lbrace\cdot;\cdot\rbrace$ denotes the anticommutator operation, $\delta_{k\ell}$ is the Kronecker delta function, $\alpha_0$ represents a $4\times 4$ matrix operator that acts on the wave function of a particle and corresponds to the energy of the particle, and the indices $k,\ell$ run from 1 to 3.

The Dirac equation combined special relativity and quantum mechanics and predicted the existence of antiparticles, which was later confirmed experimentally (see \cite{P2}). The equation was initially met with skepticism by many physicists, including Albert Einstein, who had doubts about its mathematical consistency. However, the Dirac equation was soon embraced as a major breakthrough in the development of quantum mechanics, and it paved the way for the development of quantum field theory. In addition to its theoretical significance, the Dirac equation has had a wide range of practical applications, particularly in condensed matter physics, where it has been used to describe the behavior of electrons in solids. It has also been used in high-energy particle physics, where it forms the basis of the standard model of particle physics, for more details see \cite{PD,P1,P2,P3,P4,P5}. These references should provide a good starting point for anyone interested in learning more about the Dirac equation and its applications in physics.\\

On the other hand, the Bopp-Podolsky ((BP) for short) theory (or electrodynamics) (see \cite{B1}), developed by Bopp \cite{bopp}, and independently by Podolsky \cite{podl} is a
second-order gauge theory for the electromagnetic field. It was introduced to solve the so-called "infinity problem" that appears in the classical Maxwell theory. In fact, by the well-known Poisson equation (or Gauss law), the electrostatic potential $\phi$ for a given charge distribution whose density is $\rho$ satisfies the equation
\begin{equation}\label{1.2}
    -\triangle \phi=\rho,\ \ \text{on}\  \mathbb{R}^3.
\end{equation}
If $\rho=4\pi\delta_{x_0}$, ($x_0\in \mathbb{R}^3$), then $\mathcal{G}(x-x_0)$, with $\displaystyle{\mathcal{G}(x):=\frac{1}{\vert x\vert}}$, is the fundamental solution of \eqref{1.2} and $$\displaystyle{\mathcal{E}_M}(\mathcal{G}):=\frac{1}{2}\int_{\mathbb{R}^3}\vert \nabla\mathcal{G} \vert^2dx=+\infty$$
its electrostatic energy. Thus, in the Bopp-Podolsky theory, the   equation \eqref{1.2} is replaced by
\begin{equation}\label{1.1}
-\triangle\phi+a^2\triangle^2 \phi=\rho,\ \ \text{on}\  \mathbb{R}^3.
\end{equation} Therefore, In this case, if $\rho=4\pi\delta_{x_0}$, ( $x_0\in \mathbb{R}^3$), we are able to know explicitly the solution of the  equation \eqref{1.1} and to see that its energy is finite or not. Fortunately, in \cite{jde}  P. d'Avenia and G. Siciliano proved that $\mathcal{K}(x-x_0)$ with, $\displaystyle{\mathcal{K}(x):=\frac{1-e^{-\frac{\vert x\vert}{a}}}{\vert x\vert}}$, is the fundamental solution of the equation
$$-\triangle\phi+a^2\triangle^2 \phi=4\pi\delta_{x_0},\ \ \text{on}\  \mathbb{R}^3.$$
The solution of the previous equation  has no singularity in $x_0$ since it satisfies 
$$\lim\limits_{x\rightarrow x_0}\mathcal{K}(x-x_0)=\frac{1}{a},$$
and its energy is
$$\mathcal{E}_{BP}(\mathcal{K}):=\frac{1}{2}\int_{\mathbb{R}^3}\vert \nabla \mathcal{K}\vert^2dx+\frac{a^2}{2}\int_{\mathbb{R}^3}\vert \triangle \mathcal{K}\vert^2dx<+\infty.$$
For more details about this subject see \cite[Section 2]{jde}. Moreover, the (BP) theory may be interpreted as an effective theory for short distances
(see \cite{B7}) and for large distances, it is experimentally indistinguishable from the Maxwell one. Thus, the Bopp-Podolsky parameter $a > 0$, which has dimension of the inverse of mass, can be interpreted as a cut-off distance or can be linked to an effective radius for the electron. For more physical details about the (BP) electrodynamics, see \cite{B1,B2,B3,B4,B5,B6}.\\ 

From a physical point of view, the relationship between the Dirac equation and the Bopp-Podolsky electrodynamics has been studied by several authors (see \cite{BP2} and their references). One approach is to modify the four-potential in the Dirac equation to include the (BP) potential term, which leads to a modified Dirac equation that includes the effects of the BP electrodynamics. This modified Dirac equation can then be used to study the behavior of particles in strong electromagnetic fields, such as those found in high-energy physics and astrophysics.
The modified Dirac equation in the (BP) theory has been used to investigate the effects of spontaneous emission, as well as other new physical phenomena that arise from the (BP) potential term. However, it is important to note that the relationship between the Dirac equation and the (BP) electrodynamics is still an area of active research, and there is ongoing work to better understand the nature of this relationship and its implications for our understanding of quantum mechanics and electromagnetism, for more details see \cite{BP1,BP2,BP3,BP4,BP5}.\\

Mathematically, there are many papers focused on the existence of  solutions for Dirac equations coupling to the electromagnetic field (Maxwell theory) under various hypotheses on the external field and nonlinearity, see \cite{Zh,DM1,DM2,DM3,DM4,DM5,DM6,DM7,DM8,DM9,DM10,DM11} and their references. For example, in \cite{Zh}, Jian Zhang, Wen Zhang, and Xianhua Tang studied the following  nonlinear Maxwell-Dirac system:
\begin{equation}\label{MD}
\left\lbrace
\begin{array}{rll}
  \displaystyle{ -i\sum_{k=1}^{3}\alpha_{k}\partial_{k}u+[V(x)+q]\beta u+wu-\phi u}&=f(x,u),  \ \ &\text{in}\ \mathbb{R}^3,  \\
   \ & \ & \ \\
   -\triangle\phi&=4\pi \vert u\vert^2,\ \ &  \text{in}\ \mathbb{R}^3, 
\end{array}
\tag{$\mathcal{MD}$}
\right.
\end{equation}
where $u:\mathbb{R}^3\rightarrow \mathbb{C}^4$, $\phi:\mathbb{R}^3\rightarrow \mathbb{R}$ , $q=\frac{mc}{\hbar}$, and  $w=\frac{\theta}{c\hbar}$, $\theta\in \mathbb{R}$. Precisely, under suitable assumptions on the potential function $V(x)$ and the nonlinear term $f(x,u)$, they proved the existence of infinitely
many solutions of the system \eqref{MD}.\\

On the other side, to the best of our knowledge, mathematically, the Bopp-Podolsky  theory has appeared recently in the paper of P. d'Avenia and G. Siciliano \cite{jde}. In this last reference, the authors coupled the Schr\"odinger equations with the (BP) electrodynamics. More precisely, they studied the following Schr\"odinger-Bopp-Podolsky system:
\begin{equation}\label{SBP}
\left\lbrace
\begin{array}{rll}
   -\triangle u+wu+s^2\phi u&=\vert u\vert^{p-2}u,  \ \ &\text{in}\ \mathbb{R}^3,  \\
   \ & \ & \ \\
   -\triangle\phi+a^2\triangle^2 \phi&=4\pi u^2,\ \ &  \text{in}\ \mathbb{R}^3, 
\end{array}
\tag{$\mathcal{SBP}$}
\right.
\end{equation}
where $u,\phi:\mathbb{R}^3\rightarrow \mathbb{R}$, $a>0$ is the Bopp-Podolsky parameter, and  $s\neq 0$. Moreover, in \cite{jde}, the authors proved existence and nonexistence results depending on the parameters $s,p$ and they showed
that in the radial case, the solutions funded tend to solutions of the
classical Schr\"odinger-Poisson system as $a\rightarrow 0$. After the pioneering work by P. d'Avenia and G. Siciliano \cite{jde}, system \eqref{SBP} began to attract the attention of many mathematicians; see for instance \cite{ZCC,GK,CT,LPT,YCL,AS,EH,HE,EH1,AG,CRT,MS,Z1,FS,LC,TY,LT1} for positive solutions and \cite{Z2,HWT,WCL,HB} for sign-changing
solutions.\\

Motivated by the physics and mathematics background of Dirac equations and Bopp-Podolsky theory, in this paper, we study the Dirac equations coupled with (BP) electrodynamics. Precisely, we study the existence of solutions for the following  Dirac-Bopp-Podolsky system

\begin{equation}\label{DBP}
\left\lbrace
\begin{array}{rll}
  \displaystyle{ -i\sum_{k=1}^{3}\alpha_{k}\partial_{k}u+[V(x)+q]\beta u+wu-\phi u}&=f(x,u),  \ \ &\text{in}\ \mathbb{R}^3,  \\
   \ & \ & \ \\
   -\triangle\phi+a^2\triangle^2 \phi&=4\pi \vert u\vert^2,\ \ &  \text{in}\ \mathbb{R}^3, 
\end{array}
\tag{$\mathcal{DBP}$}
\right.
\end{equation}
where $u:\mathbb{R}^3\rightarrow \mathbb{C}^4$, $\phi:\mathbb{R}^3\rightarrow \mathbb{R}$ , $q=\frac{mc}{\hbar}$, $w=\frac{\theta}{c\hbar}$, $\theta\in \mathbb{R}$,  $a>0$ is the Bopp-Podolsky parameter, $V(x)$ is a potential function, and $f(x, u)$ is the interaction term (nonlinearity).

For what concerns the nonlinearity reaction  term $f:\mathbb{R}^3\times \mathbb{C}^4\rightarrow \mathbb{R}$, we assume that $f$ is a measurable in the first variable $x\in \mathbb{R}^3$, continuous in the second variable $u\in \mathbb{C}^4$, and satisfies the following assumptions:
\begin{enumerate}
\item[$(f_1)$] $f(x, u) \in C(\mathbb{R}^3\times\mathbb{C}^4,\mathbb{R}_+)$, $f(x, u)$ is 1-periodic in $x_k$, $k = 1, 2, 3$ and $F(x,u)\geq 0$.
\item[$(f_2)$] $\displaystyle{\frac{F(x,u)}{ \vert u\vert^{2}}}\longrightarrow +\infty$, as $\vert u\vert\longrightarrow +\infty$ uniformly in $x \in \mathbb{R}^3$.
\item[$(f_3)$] $f(x,u)=o(\vert u\vert)$ as $\vert u\vert\longrightarrow 0$ uniformly in $x\in\mathbb{R}^3$.
\item[$(f_4)$]  $\displaystyle{\widetilde{F}(x,u)=\frac{1}{2}f(x,u)u-F(x,u)}>0$, for $\vert u\vert$ large and there exist  constants
$\sigma>\frac{3}{2},\ \widetilde{C}>0$ and $r_0>0$, such that
$$\vert f(x,u)\vert^{\sigma}\leq \widetilde{C} \vert u\vert^{\sigma}\widetilde{F}(x,u),\ \ \forall\ (x,u)\in \mathbb{R}^3\times\mathbb{C}^4,\ \  \vert u\vert \geq r_0.$$
\end{enumerate}
Before stating our main result, we need the following hypotheses on the potential function $V(x)$ and on the parameters $w,q$.
\begin{enumerate}
    \item[$(A_1)$] $w\in (-q,q)$.
\item[$(A_2)$]$V\in C^{1}(\mathbb{R}^3,\mathbb{R}_+)$, and $V(x)$ is 1-periodic in $x_k, k = 1, 2, 3$.
\end{enumerate}

Our main result is summarized in the following theorem:
\begin{thm}\label{thm1}
Suppose that the hypotheses $(A_1)-(A_2)$ and  $(f_1)-(f_4)$ hold. Then, system \eqref{DBP} admits at least one pair of solutions $(v,\phi_v)$. 
\end{thm}

From a mathematical point of view, as far as we know, this is the first work that dealt with the existence of solutions for  Dirac-Bopp-Podolsky systems. The main feature of our  is that the Dirac operator not only has unbounded positive continuous spectrum but also has unbounded negative continuous spectrum, and the corresponding energy functional is strongly indefinite. On the other hand, the main difficulty when dealing with this problem is the lack of compactness of Sobolev embedding, hence our problem poses more challenges in the calculus of variation. In order to overcome these difficulties, we will turn to the linking and concentration compactness arguments \cite{th2}.\\

The paper is organized as follows. In Sect. 2, we give the variational setting to study the system \eqref{DBP}. The Sect. 3 is devoted to proving the  linking structure of the associated energy functional with system \eqref{DBP}. In Sect. 4, we discuss the boundedness of the Cerami sequence.  Finally, we prove our main result Theorem \ref{thm1}.

\section{Variational setting}
In this Section, we give the variational setting. 
Below by $\Vert \cdot\Vert_r$ we denote the usual $L^r$-norm. For the reader's convenience, let
$$A:=-i\sum_{k=1}^{3}\alpha_{k}\partial_{k}u+[V+q]\beta$$
be the Dirac operator. It is worth mentioning that $A$ is a selfadjoint operator acting on $L^{2}(\mathbb{R}^3,\mathbb{C}^4)$ with $\mathcal{D}(A)=H^{1}(\mathbb{R}^3,\mathbb{C}^4)$ \cite[Lemma 7.2 a]{th1}. Let $\vert A\vert$
and $\vert A\vert^{\frac{1}{2}}$ denote respectively the absolute value of $A$ and the square root of $\vert A\vert$,
and let $\lbrace \mathcal{F}_{\lambda}:\ -\infty\leq \lambda\leq +\infty\rbrace$ be the spectral family of $A$. Set $U=id-\mathcal{F}_0-\mathcal{F}_{0-}$.
Then $U$ commutes with $A$, $\vert A\vert$ and $\vert A\vert^{\frac{1}{2}}$, and $A=U\vert A\vert$ is the polar decomposition
of $A$. Let $\sigma(A)$ and $\sigma_c(A)$ be, respectively, the spectrum and  the continuous spectrum of $A$.
In order to construct a suitable variational setting for the system \eqref{DBP}, we need some notions and results.
\begin{lemma}\label{lemma4}
   Assume that $(A_2)$ holds. Then
   $$\sigma(A)=\sigma_c(A)\subset(-\infty,-q]\cap[q,+\infty),$$
   and 
   $$\inf\sigma(\vert A\vert)\leq q+\sup_{x\in \mathbb{R}^3} V(x)$$
\end{lemma}
It follows that the space $L^{2}(\mathbb{R}^3,\mathbb{C}^4)$ possesses the orthogonal decomposition:
$$L^{2}(\mathbb{R}^3,\mathbb{C}^4)=L^+\oplus L^-,\ u=u^++u^-$$
such that $A$ is negative definite on $L^-$ and positive definite on $L^+$. Let $\mathbb{E}:=\mathcal{D}(\vert A\vert^{\frac{1}{2}})$ be the domain of $\vert A\vert^{\frac{1}{2}}$. We define on $\mathbb{E}$ the following inner product 
$$\left\langle u,v \right\rangle=\left\langle \vert A\vert^{\frac{1}{2}}u,\vert A\vert^{\frac{1}{2}}v \right\rangle_2+w\left\langle u,v^+-v^- \right\rangle_2$$
where $\langle \cdot,\cdot\rangle_2$ denote the usual $L^2$ inner product. Therefore, the induced norm on $\mathbb{E}$ is 
$$\Vert u\Vert:=\left(\Vert \vert A\vert^{\frac{1}{2}}u\Vert_2^2+w\left(\Vert u^+\Vert_2^2-\Vert u^-\Vert_2^2\right)\right)^{\frac{1}{2}}.$$
Anyone can check that $\mathbb{E}$ 
possesses the following decomposition
$$\mathbb{E}=\mathbb{E}^+\oplus\mathbb{E}^-\ \ \text{and}\ \ \mathbb{E}^{\pm}=\mathbb{E}\cap L^{\pm}.$$
Thus,
$$
\left\lbrace
\begin{array}{l}
Au=-\vert A\vert u,\ \text{for all}\ u\in \mathbb{E}^-,\\
\ \\
Au=\vert A\vert u,\ \text{for all}\ u\in \mathbb{E}^+,\\
\ \\
\text{and}\\
\ \\
u=u^++u^-,\ \text{for all}\ u\in \mathbb{E}.
\end{array}
\right.
$$
Hence $\mathbb{E}^-$ and $\mathbb{E}^+$ are orthogonal with respect to both $\langle \cdot,\cdot\rangle_2$  and $\langle \cdot,\cdot\rangle$  inner products.\\

In what follows, we give a crucial result which is the compact and continuous embedding of the new space $\mathbb{E}$ into Lebesgue spaces.
\begin{lemma}[\cite{Zh}]\label{lemma5}\ \\
    Assume that $(A_1)-(A_2)$ are satisfied. Then, $\mathbb{E}=H^{\frac{1}{2}}(\mathbb{R}^3,\mathbb{C}^4)$, with equivalent norms. Moreover, we have that 
    \begin{enumerate}
        \item[$(1)$] the embedding $\mathbb{E}\hookrightarrow L^{p}$ is continuous for all $p\in[2,3]$;
        \item[$(2)$] the embedding $\mathbb{E}\hookrightarrow L^{p}_{\text{loc}}$ is compact for all $p\in[1,3)$;
        \item[$(3)$] $(q-\vert w\vert)\Vert u\Vert^2_2\leq \Vert u\Vert^2$, for all $u\in \mathbb{E}$.
    \end{enumerate}
\end{lemma}

An important fact involving system \eqref{DBP} is that this class of system can be transformed into a
Dirac equation with a nonlocal term (see\cite{jde}), which allows us to use variational
methods. Effectively, as we mentioned in Section 1, in \cite{jde}, P. d'Avenia and G. Siciliano, by the Lax-Milgram Theorem, proved that for a given $u\in H^1(\mathbb{R}^3,\mathbb{C}^4)$, there exists a unique $\phi_u\in \mathcal{D}$ such that
$$-\triangle\phi_u+a^2\triangle^2 \phi_u=4\pi \vert u\vert^2,\ \text{in}\ \mathbb{R}^3,$$
where $\mathcal{D}$ is the completion of $C_0^{\infty}(\mathbb{R}^3)$ with respect to the norm $\Vert \cdot\Vert_{\mathcal{D}}$ induced by the inner product
$$\langle \psi,\varphi\rangle_{\mathcal{D}}:=\int_{\mathbb{R}^3}\nabla \psi\nabla\varphi+a^2\triangle\psi\triangle\varphi dx.$$
Otherwise, $$\mathcal{D}:=\left\lbrace \phi\in \mathcal{D}^{1,2}(\mathbb{R}^3):\ \triangle \phi\in L^2(\mathbb{R}^3)\right\rbrace,$$
with  
$$\mathcal{D}^{1,2}(\mathbb{R}^3):=\left\lbrace \phi\in L^{6}(\mathbb{R}^3):\ \nabla \phi\in L^2(\mathbb{R}^3)\right\rbrace.$$
The space $\mathcal{D}$ is an Hilbert space continuously embedded into $\mathcal{D}^{1,2}(\mathbb{R}^3)$, $L^6(\mathbb{R}^3)$ and $L^\infty (\mathbb{R}^3).$\\

Arguing as in \cite{jde}, we prove that the unique solution has the following expression $\displaystyle{\phi_u:=\mathcal{K}*\vert u\vert^2=\frac{1-e^{-\frac{\vert x\vert}{a}}}{\vert x\vert}*\vert u\vert^2}$, for all $u\in H^1(\mathbb{R}^3,\mathbb{C}^4)$, and it   verifies the following properties:
\begin{lemma}\label{lem1}\ \\
For any $u\in H^1(\mathbb{R}^3,\mathbb{C}^4)$, we have: 
\begin{enumerate}
    \item[$(1)$] for every $y\in \mathbb{R}^3$, $\phi_{u(\cdot+y)}=\phi_{u}(\cdot+y)$;
    \item[$(2)$] $\phi_u\geq 0$;
    \item[$(3)$] for every $r\in(3,+\infty]$,\ $\phi_u\in L^r(\mathbb{R}^3)\cap C_0(\mathbb{R}^3)$;
    \item[$(4)$]for every $r\in(\frac{3}{2},+\infty]$, $\nabla\phi_u=\nabla\left(\frac{1-e^{-\frac{\vert x\vert}{a}}}{\vert x\vert}\right)*\vert u\vert^2\in L^{r}(\mathbb{R}^3)\cap C_0(\mathbb{R}^3)$;
    \item[$(5)$]$\phi_u\in\mathcal{D}$;
    \item[$(6)$]$\Vert \phi_u\Vert_6\leq C\Vert u\Vert^2$;
    \item[$(7)$] $\phi_u$ is the unique minimizer of the functional
    $$E(\phi):=\frac{1}{2}\Vert \nabla \phi\Vert^2_2+\frac{a^2}{2}\Vert \triangle \phi\Vert^2_2-\int_{\mathbb{R}^3}\phi \vert u\vert^2dx,\ \ \text{for all}\ \phi \in\mathcal{D};$$
    \item[$(8)$]if $v_n\rightharpoonup v$ in $H^1(\mathbb{R}^3,\mathbb{C}^4)$, then $\phi_{v_n}\rightharpoonup \phi_v$ in $\mathcal{D}$ and $\displaystyle{\int_{\mathbb{R}^3}\phi_{u_n}\vert u_n\vert^2 dx\rightarrow \int_{\mathbb{R}^3}\phi_{u}\vert u\vert^2 dx };$
   \item[$(9)$] $\phi_{tu}=t^2\phi_u$, for all $t\in \mathbb{R}_+$;
   \item[$(10)$]$\displaystyle{\int_{\mathbb{R}^3}\phi_u \vert u\vert^2dx=\int_{\mathbb{R}^3}\int_{\mathbb{R}^3}\frac{1-e^{-\frac{\vert x-y\vert}{a}}
   }{\vert x-y\vert}\vert u(x)\vert^2\vert u(y)\vert^2dxdy\leq \frac{1}{a}\Vert u\Vert_2^4}.$
\end{enumerate}
\end{lemma}

Therefore, the pair $(u,\phi)\in H^1(\mathbb{R}^3,\mathbb{C}^4)\times \mathcal{D}$ is a solution of \eqref{DBP} if, and only if, for all $u\in H^1(\mathbb{R}^3,\mathbb{C}^4)$, we have that $\phi=\phi_u$ is a weak solution of the following nonlocal problem
\begin{equation}\label{P}
   \displaystyle{ -i\sum_{k=1}^{3}\alpha_{k}\partial_{k}u+[V(x)+q]\beta u+wu-\phi_u u}=f(x,u),  \ \ \text{in}\ \mathbb{R}^3. \tag{$\mathcal{P}$}
\end{equation}

Next, we would like to mention that, from assumptions $(f_1)-(f_2)$  and Lemma \ref{lem1}, the existence of solutions for problem \eqref{P} can be made via
variational methods. In particular,
 the corresponding energy functional to problem \eqref{P} is $J:\mathbb{E}\longrightarrow \mathbb{R}$, which is defined by 
\begin{align}\label{J}
J(u):=\frac{1}{2} \left(\Vert u^+\Vert^2-\Vert u^-\Vert^2\right)-\Gamma(u)-\int_{\mathbb{R}^3}F(x,u)dx,\ \text{ for all}\ u\in\mathbb{E},
\end{align}
where
$$
\Gamma(u):=\frac{1}{4}\int_{\mathbb{R}^3}\phi_u \vert u\vert^2dx=\frac{1}{4}\int_{\mathbb{R}^3}\int_{\mathbb{R}^3}\frac{1-e^{-\frac{\vert x-y\vert}{a}}
   }{\vert x-y\vert}\vert u(x)\vert^2\vert u(y)\vert^2dxdy.
$$
The functional $J$ belongs to $C^1\left(\mathbb{E},\mathbb{R}\right)$ and a standard argument
shows that critical points of $J$ are solutions of problem \eqref{P} (see \cite{th1,th2}).\\

Once we apply variational
methods on the problem \eqref{P}, we find  the following nonlocal term $\displaystyle{\int_{\mathbb{R}^3}\phi_u \vert u\vert^2dx}$ which is homogeneous of degree 4. Thus, the natural  corresponding Ambrosetti-Rabinowitz
condition on $f(x,u)$ is the following:
\begin{enumerate}
    \item[(AR)] There exists $\Theta>4$ such that $0<\Theta F(x,u)\leq f(x,u)\vert u\vert$, for a.a. $x\in \mathbb{R}^3$ and all $\vert u\vert >0$.
\end{enumerate}
Therefore, our assumption $(f_4)$ is weaker than the (AR) condition. Indeed, the following examples satisfies assumptions $(f_1)-(f_2)$ but not the (AR) condition:
\begin{enumerate}
    \item[$(1)$] $f(x,u)=b(x)\vert u\vert \ln(1+\vert u\vert)$, where $b\in C(\mathbb{R}^3,\mathbb{R})$ and is $1$-periodic in $x_k,\ k=1,2,3$,
    \item[$(2)$] $\displaystyle{F(x,u)=b(x)\left(\vert u\vert^{\nu}+(\nu-2)\vert u\vert^{\nu-\varepsilon}\text{sin}^2\left(\frac{\vert u\vert^{\varepsilon}}{\varepsilon} \right)\right)}$, where $\nu>2$ and $0<\varepsilon<\frac{6-\nu}{2}$.
\end{enumerate}
\section{The linking structure}
In this section, we discuss the linking structure of the functional $J$.
First, let $r>0$, set $B_r:=\lbrace  u\in\mathbb{E}:\ \Vert u\Vert \leq r\rbrace$ and $S_r:=\lbrace u\in\mathbb{E}:\ \Vert u\Vert = r\rbrace $. Let us observe that from assumptions $(f_1)-(f_4)$,  for any $\varepsilon>0$, there exist positive constants $r_\varepsilon>0$ and $C_\varepsilon$ such that
\begin{equation}\label{eq994}
    \left\lbrace
\begin{array}{l}
   \vert f(x,u)\vert\leq \varepsilon\vert u\vert,\ \text{for all}\ 0\leq \vert u\vert \leq r_\varepsilon\ \text{and all}\ x\in \mathbb{R}^3,\\
   \ \\
   \vert f(x,u)\vert\leq \varepsilon\vert u\vert+C_\varepsilon\vert u\vert^{p-1},\ \text{for all}\ (x,u)\in \mathbb{R}^3\times\mathbb{C}^4,\\
   \ \\
   \text{and}\\
   \ \\
   \vert F(x,u)\vert\leq \varepsilon\vert u\vert^2+C_\varepsilon\vert u\vert^{p},\ \text{for all}\ (x,u)\in\mathbb{R}^3\times\mathbb{C}^4,
\end{array}
    \right.
\end{equation}
where $p\in (2,3)$.\\

By a standard argument of \cite{th2} and arguing as in \cite[Lemma 3.1]{Zh}, we get the following result.
\begin{lemma}\label{lemma8}
    Under the assumption of Theorem \ref{thm1}, we have 
    \begin{enumerate}
        \item[$(1)$] $\Gamma$ and  $\displaystyle{\int_{\mathbb{R}^3}F(x,u)dx} $  are non-negative, weakly sequentially lower semi-continuous;
        \item[$(2)$] $\Gamma^{'}$ and $\displaystyle{\int_{\mathbb{R}^3}f(x,u)dx} $ are weakly sequentially continuous;
        \item[$(3)$] there exists $\xi>0$ such that
        for any $c>0$
        $$\Vert u\Vert\leq \xi \Vert u^+\Vert,\ \ \text{for all}\ u\in J_c,$$
        where $J_c:=\lbrace u\in \mathbb{E}:\ \ J(u)\geq c \rbrace.$
    \end{enumerate}
\end{lemma}
\begin{prop}\label{prop7}
    Assume that assumptions of Theorem \ref{thm1} are satisfied. Then, there exists $r>0$ such that 
    $$\rho:=\inf J(S_r\cap \mathbb{E}^+)>0.$$
\end{prop}
\begin{proof}
From Lemma \ref{lemma5}, we have 
$$\Vert u\Vert_p^p\leq c_p\Vert u\Vert^p,\ \text{for all}\  u\in\mathbb{E}\ \text{and all}\ p\in(2,3).$$
It follows, by \eqref{eq994} and Lemma \ref{lem1}-(10), that
\begin{align*}
 J(u) & = \frac{1}{2} \left(\Vert u^+\Vert^2-\Vert u^-\Vert^2\right)-\Gamma(u)-\int_{\mathbb{R}^3}F(x,u)dx \\
 & \geq \frac{1}{2}\Vert u\Vert^2-C_1\Vert u\Vert^4-\varepsilon C_2 \Vert u\Vert^2-C_{\varepsilon}c_p\Vert u\Vert^p\\
 &= \left(\frac{1}{2}-\varepsilon C_2\right)\Vert u\Vert^2-C_1\Vert u\Vert^4-C_{\varepsilon}c_p\Vert u\Vert^p.
\end{align*}
Choosing $\varepsilon=\frac{1}{4C_2}$ in the previous inequality and using the fact that $p\in (2,3)$, we can find $r>0$ sufficiently small such that
$$J(u)>0,\ \text{for all}\ u\in S_r.$$
Thus, the desired result holds.
\end{proof}
As a consequence of Lemma \ref{lemma4}, we have 
$$q\leq \overline{\Lambda}\leq q +\sup_{\mathbb{R}^3}V,$$
where $\overline{\Lambda}:=\inf \sigma(A)\cap [0,+\infty).$\\
Next, we set $\overline{\mu}:=2\overline{\Lambda}$ and we take a number $\mu$ satisfying
\begin{equation}\label{eq888}
\overline{\Lambda}\leq \mu\leq \overline{\mu}.
\end{equation}
Since the operator $A$ is invariant under the action of $\mathbb{Z}^3$ ( by $(A_2)$), the subspace $Y_0:=\left(\mathcal{F}_{\mu}-\mathcal{F}_0\right)L^2$
is infinite-dimensional, and
\begin{equation}\label{eq991}
\overline{\Lambda}\Vert u \Vert^2_2\leq \Vert u\Vert^2\leq \mu\Vert u \Vert^2_2,\ \ \text{for all}\ u\in Y_0.
\end{equation}
For any finite-dimensional subspace $Y$ of $Y_0$, we set $\mathbb{E}_Y:=\mathbb{E}^-\oplus Y$. 
\begin{prop}\label{prop9}
 Assume that assumptions of Theorem \ref{thm1} be satisfied. Then, for any finite-dimensional subspace $Y$ of $Y_0$, $\sup J(\mathbb{E}_Y)<+\infty$, and there is $R_Y>0$ such that
 $$J(u)<\inf J(B_{\delta}),\ \text{for all}\ \ u\in E_Y\ \ \text{with}\ \Vert u\Vert\geq R_Y.$$
\end{prop}
\begin{proof}
It is sufficient to show that $J(u)\longrightarrow -\infty$   as $u\in\mathbb{E}_Y$, $\Vert u\Vert\longrightarrow +\infty$. Arguing
indirectly, assume that for some sequence $\lbrace u_n\rbrace_{n\in \mathbb{N}}$ with $\Vert u_n\Vert\longrightarrow +\infty$, there is $M>0$ such that $J(u_n)\geq -M$ for all $n\in \mathbb{N}$. Then, setting $v_n:=\frac{u_n}{\Vert u_n\Vert}$, we have $\Vert v_n\Vert=1,v_n\rightharpoonup v,v_n^-\rightharpoonup v^-$, and $v_n^+\rightarrow v^+\in Y$.\\
Using Lemma \ref{lemma8}-(1), we find that
\begin{equation}\label{eq990}
\frac{1}{2} \left(\Vert v_n^+\Vert^2-\Vert v_n^-\Vert^2\right)\geq \frac{J(u_n)}{\Vert u_n\Vert^2} \geq \frac{-M}{\Vert u_n\Vert^2}, 
\end{equation}
which gives that 
$$\frac{1}{2} \Vert v_n^-\Vert^2\leq \frac{1}{2} \Vert v_n^+\Vert^2+\frac{M}{\Vert u_n\Vert^2}.$$
Thus, $v^+\not\equiv 0.$\\
From assumption $(f_2)$, there is $r>0$ such that 
\begin{equation}\label{eq889}
 F(x,u)\geq \overline{\mu}\vert u\vert^2,\ \text{if}\ \vert u\vert \geq r.  
\end{equation}
It follows, by \eqref{eq888} and \eqref{eq991}, that
\begin{align}\label{eq887}
    \Vert v^+\Vert^2-\Vert v^-\Vert^2-\overline{\mu}\int_{\mathbb{R}^3} \vert v\vert^2dx& = \Vert v^+\Vert^2-\Vert v^-\Vert^2-\overline{\mu}\Vert v\Vert^2_2\nonumber\\
    & \leq \mu\Vert v^+\Vert^2_2-\Vert v^-\Vert^2-\overline{\mu}\Vert v^-\Vert^2_2-\overline{\mu}\Vert v^+\Vert^2_2\nonumber\\
    & \leq -\left((\overline{\mu}-\mu)\Vert v^+\Vert^2_2+\Vert v^-\Vert^2\right)\nonumber\\
    & <0.
\end{align}
Therefore, there is a bounded domain $\Omega\subset \mathbb{R}^3$ such that
\begin{equation}\label{eq886}
   \Vert v^+\Vert^2-\Vert v^-\Vert^2-\overline{\mu}\int_{\Omega} \vert v\vert^2dx <0.
\end{equation}
Using \eqref{eq889} and Lemma \eqref{lemma8}-(1), we see that
\begin{align*}
  \frac{J(u_n)}{\Vert u_n\Vert^2}&\leq   \frac{1}{2} \left(\Vert v_n^+\Vert^2-\Vert v_n^-\Vert^2\right)-\int_{\Omega}\frac{F(x,u_n)}{\Vert u_n\Vert^2}dx\\
  & \leq \frac{1}{2} \left(\Vert v_n^+\Vert^2-\Vert v_n^-\Vert^2-\overline{\mu}\int_{\mathbb{R}^3} \vert v_n\vert^2dx\right)-\int_{\Omega}\frac{F(x,u_n)-\frac{\overline{\mu}}{2}\vert v_n\vert^2}{\Vert u_n\Vert^2}dx\\
  & \leq \frac{1}{2} \left(\Vert v_n^+\Vert^2-\Vert v_n^-\Vert^2-\overline{\mu}\int_{\mathbb{R}^3} \vert v_n\vert^2dx\right)+\frac{\overline{\mu}r^2\vert \Omega\vert}{2\Vert u_n\Vert^2},
\end{align*}
where $\vert \Omega\vert$ denotes the Lebesgue measure of $\Omega$. Thus, \eqref{eq990} and \eqref{eq886}
imply that 
\begin{align*}
  0&\leq \lim\limits_{n\rightarrow +\infty} \left(\frac{1}{2}\Vert v_n^+\Vert^2-\frac{1}{2}\Vert v_n^-\Vert^2-\int_{\mathbb{R}^3}\frac{F(x,u_n)}{\Vert u_n\Vert^2}\right)\\
   & \leq \frac{1}{2} \left(\Vert v^+\Vert^2-\Vert v^-\Vert^2-\overline{\mu}\int_{\mathbb{R}^3} \vert v\vert^2dx\right)\\
   & <0,
\end{align*}
which is a contradiction.
\end{proof}

As a consequence, we have the following result.
\begin{prop}\label{prop8}
    Under the assumptions of Theorem \ref{thm1}, letting $e\in Y_0$ such that $\Vert e\Vert=1$, there exists $R_0>r>0$ such that 
    $$\sup J(\partial Q)\leq \rho,$$
    where $Q:=\lbrace u=u^-+te:\ \ u^-\in \mathbb{E}^-, t\geq 0,\ \Vert u\Vert\leq R_0\rbrace$ and $\rho$ given in Proposition \ref{prop7}.
\end{prop}
\section{ The Cerami sequence}
In this section,
we consider the boundedness of the Cerami sequence. Firstly, recall that a sequence $\lbrace u_n\rbrace_{n\in \mathbb{N}}\subset \mathbb{E}$ is a Cerami sequence at the level $c$ ($(C)_c$-sequence
for short) for the functional $J$ if $$J(u_n)\longrightarrow c\ \ \text{and}\ \ (1+\Vert u_n\Vert) J^{'}(u_n)\longrightarrow 0,\ \ \text{as}\ \ n\longrightarrow +\infty.$$

Before starting the main goal of this section, we give a crucial result on the nonlocal term $\Gamma$.
\begin{lemma}\label{lemma1}
  For any $u\in \mathbb{E}\setminus\lbrace 0\rbrace$, there exists $C > 0$ such that  
  $$\Gamma^{'}(u)u>0\ \ \text{and}\ \ \Vert \Gamma^{'}(u)\Vert_{\mathbb{E}^*}\leq C\left(\sqrt{\Gamma^{'}(u)u}+\Gamma^{'}(u)u\right).$$
\end{lemma}
\begin{proof}
    It's clear that $$\Gamma^{'}(u)u=4\Gamma (u)>0,\ \ \text{for all}\ u\in \mathbb{E}\setminus\lbrace 0\rbrace.$$
This shows the first inequality. So, 
it remains to prove the second inequality. Since $\Gamma$ is the unique nonlocal term in the functional $J$, from the argument in
Ackermann \cite{AK}, it yields that
\begin{align}\label{eq996}
 \int_{\mathbb{R}^3}\phi_u \vert v\vert^2dx&=\int_{\mathbb{R}^3}\left(\frac{1-e^{-\frac{\vert x\vert}{a}}
   }{\vert x\vert}*\vert u\vert ^2\right)\vert v\vert^2dx\nonumber\\
&\leq C_1\left(\int_{\mathbb{R}^3}\left(\frac{1-e^{-\frac{\vert x\vert}{a}}
   }{\vert x\vert}*\vert u\vert ^2\right)\vert u\vert^2dx\int_{\mathbb{R}^3}\left(\frac{1-e^{-\frac{\vert x\vert}{a}}
   }{\vert x\vert}*\vert v\vert ^2\right)\vert v\vert^2dx\right)^{\frac{1}{2}},
\end{align}
for all $u,v\in\mathbb{E}$ and some constant $C_1>0$.\\
It follows, by Lemma \ref{lem1}-(10), H\"older inequality and Sobolev embedding theorem, that
\begin{align}\label{eq995}
  &\int_{\mathbb{R}^3}\left(\frac{1-e^{-\frac{\vert x\vert}{a}}
   }{\vert x\vert}*\vert u\vert^2 \right)\vert uv\vert dx\nonumber\\
   & \leq \left(\int_{\mathbb{R}^3}\left(\frac{1-e^{-\frac{\vert x\vert}{a}}
   }{\vert x\vert}*\vert u\vert^2 \right)\vert u\vert^2 dx\right)^{\frac{1}{2}}\left(\int_{\mathbb{R}^3}\left(\frac{1-e^{-\frac{\vert x\vert}{a}}
   }{\vert x\vert}*\vert u\vert^2 \right)\vert v\vert^2 dx\right)^{\frac{1}{2}}\nonumber\\
   & \leq C_2 \left(\int_{\mathbb{R}^3}\left(\frac{1-e^{-\frac{\vert x\vert}{a}}
   }{\vert x\vert}*\vert u\vert^2 \right)\vert u\vert^2 dx\right)^{\frac{1}{2}}\left(\int_{\mathbb{R}^3}\left(\frac{1-e^{-\frac{\vert x\vert}{a}}
   }{\vert x\vert}*\vert u\vert^2 \right)\vert u\vert^2 dx\right)^{\frac{1}{4}}\nonumber\\
   & \quad \times \left(\int_{\mathbb{R}^3}\left(\frac{1-e^{-\frac{\vert x\vert}{a}}
   }{\vert x\vert}*\vert v\vert^2 \right)\vert v\vert^2 dx\right)^{\frac{1}{4}}\nonumber\\
   & \leq C_2 \left(\int_{\mathbb{R}^3}\left(\frac{1-e^{-\frac{\vert x\vert}{a}}
   }{\vert x\vert}*\vert u\vert^2 \right)\vert u\vert^2 dx\right)^{\frac{3}{4}}\frac{1}{a}\Vert v \Vert_2^4\nonumber\\
  & \leq C_3 \left(\int_{\mathbb{R}^3}\left(\frac{1-e^{-\frac{\vert x\vert}{a}}
   }{\vert x\vert}*\vert u\vert^2 \right)\vert u\vert^2 dx\right)^{\frac{3}{4}}\Vert v \Vert.
\end{align}
Thus, 
$$\vert \Gamma^{'}(u)v\vert\leq C_3\left(\Gamma^{'}(u)u\right)^{\frac{3}{4}}\Vert v\Vert\leq C\left(\sqrt{\Gamma^{'}(u)u}+\Gamma^{'}(u)u\right)\Vert v\Vert.$$
This implies the second inequality.
\end{proof}

\begin{prop}\label{prop2}
Assume that the assumptions of Theorem \ref{thm1} hold. If $\lbrace u_n\rbrace_{n\in\mathbb{N}}\subset \mathbb{E}$ is a $(C)_{c}$-sequence for $J$, that is
$$J(u_n)\longrightarrow c\ \ \text{and}\ \ (1+\Vert u_n\Vert) J^{'}(u_n)\longrightarrow 0,\ \ \text{as}\ \ n\longrightarrow +\infty.$$ Then, $\lbrace u_n\rbrace_{n\in\mathbb{N}}$ is bounded in $ \mathbb{E}$.
\end{prop}
\begin{proof}
Let $\lbrace u_n\rbrace_{n\in\mathbb{N}}\subset \mathbb{E}$ be a $(C)_{c}$-sequence for $J$, that is
$$J(u_n)\longrightarrow c\ \ \text{and}\ \ (1+\Vert u_n\Vert) J^{'}(u_n)\longrightarrow 0,\ \ \text{as}\ \ n\longrightarrow +\infty.$$
Then,
\begin{equation}\label{eq0001}
J(u_n)\longrightarrow c \ \ \text{and}\ \ \  J^{'}(u_n)u_n\longrightarrow 0,\ \ \text{as}\ \ n\longrightarrow +\infty.
\end{equation}
By \eqref{eq0001} and assumption $(g_2)$, for $n$ sufficiently large, there exists $C_0>0$ such that
\begin{align}\label{eq0002}
C_0&\geq J(u_n)-\frac{1}{2}J^{'}(u_n)u_n =\Gamma(u_n) + \int_{\mathbb{R}^3}\widetilde{F}_+(x,u_n)dx\nonumber\\
& \geq \int_{\mathbb{R}^3}\widetilde{F}_+(x,u_n)dx.
\end{align}
Arguing by contradiction, we assume that $\Vert u_n\Vert\longrightarrow+\infty$, then $\Vert u_n\Vert \geq 1$ for $n$ is large enough. Setting $\displaystyle{v_n:=\frac{u_n}{\Vert u_n\Vert}\in \mathbb{E}}$, then   $\Vert v_n\Vert=1$ and 
\begin{equation}\label{2eq99}
    \Vert v_n\Vert_{r}\leq C_r\Vert v_n\Vert=C_r,\ \ \text{for all}\  r\in [2,2^*).
\end{equation}
Moreover, up to a subsequence, we can assume that
$$v_n\rightharpoonup v\ \text{in}\ \mathbb{R}\ \text{and}\ \ v_n(x)\rightarrow v(x)\ \text{a.e.}\ x\in \mathbb{R}^3.$$
On the other hand, for $n$ large enough, we observe that
\begin{align}\label{eq0003}
 J^{'}(u_n)(u_n^+-u_n^-)& =\Vert u_n\Vert^2-\Gamma^{'}(u_n)(u_n^+-u_n^-)-\int_{\mathbb{R}^3}f(x,u_n)(u_n^+-u_n^-)dx\nonumber\\
 & = \Vert u_n\Vert^2\left(1-\frac{\Gamma^{'}(u_n)(u_n^+-u_n^-)}{\Vert u_n\Vert^2}-\int_{\mathbb{R}^3}\frac{f(x,u_n)(u_n^+-u_n^-)}{\Vert u_n\Vert^2}dx\right)\nonumber\\
 & = \Vert u_n\Vert^2\left(1-\frac{\Gamma^{'}(u_n)(u_n^+-u_n^-)}{\Vert u_n\Vert^2}-\int_{\mathbb{R}^3}\frac{f(x,u_n)}{\vert u_n\vert}\vert v_n\vert( v_n^+-v_n^-)dx\right), 
\end{align}
which is equivalent to
\begin{align}\label{eq0004}
\frac{ J^{'}(u_n)(u_n^+-u_n^-)}{\Vert u_n\Vert^{2}}=1-\frac{\Gamma^{'}(u_n)(u_n^+-u_n^-)}{\Vert u_n\Vert^2}-\int_{\mathbb{R}^3}\frac{f(x,u_n)}{\vert u_n\vert}\vert v_n\vert( v_n^+-v_n^-)dx.
\end{align}
It follows, by \eqref{eq0001}, that
\begin{equation}\label{eq0005}
\lim\limits_{n\rightarrow+\infty}\left(\int_{\mathbb{R}^3}\frac{f(x,u_n)}{\vert u_n\vert}\vert v_n\vert( v_n^+-v_n^-)dx+\frac{\Gamma^{'}(u_n)(u_n^+-u_n^-)}{\Vert u_n\Vert^2}\right)= 1.
\end{equation}
Now,  we set for $r\geq 0$
$$\mathfrak{F}(r):=\inf\left\lbrace \widetilde{F}(x,s):\ x\in\mathbb{R}^3\ \text{and}\ s\in \mathbb{C}^4\ \text{with}\  \vert s\vert \geq r\right\rbrace. $$
By $(f_4)$, we have $$\mathfrak{F}(r)>0,\ \text{for all}\ r\ \text{large, and}\ \mathfrak{F}(r)\rightarrow +\infty\ \text{as}\ r\rightarrow+\infty.$$
For $0\leq a<b\leq +\infty$ we define
$$A_n(a,b):=\left\lbrace x\in \mathbb{R}^3:\ a\leq \vert u_n(x)\vert <b\right\rbrace $$
and
$$c_a^b:=\inf\left\lbrace \frac{\widetilde{F}(x,s)}{\vert s\vert^{2}}:\ x\in\mathbb{R}^3\ \text{and}\ s\in \mathbb{C}^4\setminus\lbrace 0\rbrace\ \text{with}\ a\leq \vert s\vert<b\right\rbrace.$$
Note that
\begin{equation}\label{eq0006}
\widetilde{F}(x,u_n)\geq c_a^b\vert u_n\vert^{2},\ \ \text{for all}\ x\in A_n(a,b).
\end{equation}
It follows, by \eqref{eq0002}, that
\begin{align}\label{eq0007}
C_0 & \geq \int_{\mathbb{R}^3}\widetilde{F}(x,u_n)dx\nonumber\\
& =\int_{A_n(0,a)}\widetilde{F}(x,u_n)dx+\int_{A_n(a,b)}\widetilde{F}(x,u_n)dx+\int_{A_n(b,+\infty)}\widetilde{F}(x,u_n)dx\nonumber\\
& \geq\int_{A_n(0,a)}\widetilde{F}(x,u_n)dx+c_a^b\int_{A_n(a,b)}\vert u_n\vert^{2}dx+\mathfrak{F}(b)\vert A_n(b,+\infty)\vert
\end{align}
for $n$ large enough.\\

Let $0<\varepsilon<\frac{1}{3}$. By assumption $(f_3)$, there exists $a_\varepsilon>0$ such that
\begin{equation}\label{eq0008}
\vert f(x,s)\vert \leq \frac{\varepsilon}{3C_r}\vert s\vert,\ \ \text{for all}\ \ \vert s\vert \leq a_\varepsilon.
\end{equation}
where $C_r$ is defined in \eqref{2eq99}.
From \eqref{eq0008} and \eqref{2eq99}, we obtain
\begin{align}\label{eq0009}
\left \vert\int_{A_n(0,a_\varepsilon)}\frac{f(x,u_n)}{\vert u_n\vert}\vert v_n\vert( v_n^+-v_n^-)dx\right\vert& \leq \int_{A_n(0,a_\varepsilon)}\frac{\vert f(x,u_n)\vert}{\vert u_n\vert}\vert v_n\vert\vert v_n^+-v_n^-\vert dx\nonumber\\
& \leq \frac{\varepsilon}{3C_2}\int_{A_n(0,a_\varepsilon)}\vert v_n\vert\vert v_n^+-v_n^-\vert dx\nonumber\\
& \leq \frac{\varepsilon}{3C_2}\int_{A_n(0,a_\varepsilon)}\vert v_n\vert^{2}dx\nonumber\\
& \leq \frac{\varepsilon}{3C_2}C_2\Vert v_n\Vert^{2}\nonumber\\
& = \frac{\varepsilon}{3},\ \ \text{for all}\ n\in \mathbb{N}.
\end{align}
Now, exploiting $(\ref{eq0007})$ and assumption $(f_4)$, we see that
$$C^{'}\geq \int_{A_n(b,+\infty)}\widetilde{F}(x,u_n)dx\geq \mathfrak{F}(b)\left\vert A_n(b,+\infty)\right\vert,$$
where $C^{'}>0$. It follows, using the fact $\mathfrak{F}(b)\rightarrow +\infty $ as $b\rightarrow+\infty$, that
\begin{equation}\label{eq00010}
\vert A_n(b,+\infty)\vert\rightarrow 0,\ \ \text{as}\ b\rightarrow+\infty,\ \ \text{uniformly in}\ n.
\end{equation}
Setting $\displaystyle{\sigma^{'}:=\frac{\sigma}{\sigma-1}}$ ( where $\sigma$ is defined in $(f_4)$ ). Since $\sigma>\frac{3}{2}$, one can check that $2\sigma^{'}\in(2,2^*)$.
Now, let $\tau\in (2\sigma^{'},2^*)$. Using \eqref{2eq99}, the H\"older inequality and \eqref{eq00010}, for $b$ large, we find that
 \begin{align}\label{eq00020}
  \left( \int_{A_n(b,+\infty)}\vert v_n\vert ^{2\sigma^{'}}dx\right)^{\frac{1}{\sigma^{'}}}
 & \leq \vert A_n(b,+\infty)\vert ^{\frac{\tau-2\sigma^{'}}{\tau\sigma^{'}}}\left( \int_{A_n(b,+\infty)}\vert v_n\vert ^{2\sigma^{'}\frac{\tau}{2\sigma^{'}}}dx\right)^{\frac{2}{\tau}} \nonumber\\
 & \leq \vert A_n(b,+\infty)\vert ^{\frac{\tau-2\sigma^{'}}{\tau\sigma^{'}}}\left( \int_{A_n(b,+\infty)}\vert v_n\vert ^{\tau}dx\right)^{\frac{2}{\tau}} \nonumber\\
 & \leq \vert A_n(b,+\infty)\vert ^{\frac{\tau-2\sigma^{'}}{\tau\sigma^{'}}}C_{\tau}\Vert v_n\Vert^{2}\nonumber\\
  & = \vert A_n(b,+\infty)\vert ^{\frac{\tau-2\sigma^{'}}{\tau\sigma^{'}}}C_{\tau}\nonumber\\
  & \leq \frac{\varepsilon}{3},\ \ \text{uniformly in}\ n.
 \end{align}
By $(f_4)$, H\"older inequality, \eqref{eq0002} and \eqref{eq00020}, we can choose $b_\varepsilon\geq r_0$ large so that
 \begin{align}\label{eq00011}
&\left\vert\int_{A_n(b_\varepsilon,+\infty)}\frac{f(x,u_n)}{\vert u_n\vert}\vert v_n\vert( v_n^+-v_n^-)dx\right\vert\nonumber\\
& \leq\int_{A_n(b_\varepsilon,+\infty)}\frac{\vert f(x,u_n)\vert}{\vert u_n\vert}\vert v_n\vert\vert v_n^+-v_n^-\vert dx\nonumber\\
 & \leq \left( \int_{A_n(b_\varepsilon,+\infty)}\left\lvert\frac{f(x,u_n)}{\vert u_n\vert}\right\rvert^{\sigma}dx\right)^{\frac{1}{\sigma}} \left( \int_{A_n(b_\varepsilon,+\infty)}(\vert v_n\vert\vert v_n^+-v_n^-\vert)^{2\sigma^{'}}dx\right)^{\frac{1}{\sigma^{'}}} \nonumber\\
 & \leq \left( \widetilde{C}\int_{A_n(b_\varepsilon,+\infty)}\widetilde{F}(x,u_n)dx\right)^{\frac{1}{\sigma}} \left( \int_{A_n(b_\varepsilon,+\infty)}\vert v_n\vert ^{2\sigma^{'}}dx\right)^{\frac{1}{\sigma^{'}}} \nonumber\\
  & \leq \frac{\varepsilon}{3},\ \ \text{uniformly in}\ n.
 \end{align}
 Next, from \eqref{eq0007}, we have
 \begin{align}\label{eq00014}
 \int_{A_n(a_{\varepsilon},b_{\varepsilon})}\vert v_n\vert^{2} dx& =\frac{1}{\Vert u_n\Vert^{2}}\int_{A_n(a_{\varepsilon},b_{\varepsilon})}\vert u_n\vert^{2} dx\nonumber\\
 &\leq \frac{C^{''}}{c_{a_{\varepsilon}}^{b_{\varepsilon}} \Vert u_n\Vert^{2}}\longrightarrow 0\ \ \text{as}\  n\longrightarrow +\infty,
 \end{align}
 where $C^{''}$ is a positive constant independent from $n$.\\
 Since $\displaystyle{\frac{f(x,s)}{ \vert s\vert}}$ is a continuous function on $a_{\varepsilon}\leq \vert s\vert\leq b_{\varepsilon}$, there exists $C>0$ depend on $a_{\varepsilon}$ and $b_{\varepsilon}$ and independent from $n$, such that
 \begin{equation}\label{eq00012}
 \vert f(x,u_n)\vert \leq C\vert u_n\vert,\ \ \text{for all}\ x\in A_n(a_{\varepsilon},b_{\varepsilon}).
 \end{equation}
 Using \eqref{eq00014} and \eqref{eq00012}, we can find $n_0>0$ 
 such that
 \begin{align}\label{eq00013}
\left\vert\int_{A_n(a_\varepsilon,b_\varepsilon)}\frac{f(x,u_n)}{\vert u_n\vert}\vert v_n\vert( v_n^+-v_n^-) dx\right\vert & \leq\int_{A_n(a_\varepsilon,b_\varepsilon)}\frac{f(x,u_n)}{\vert u_n\vert}\vert v_n\vert\vert v_n^+-v_n^-\vert\nonumber\\
 & \leq C\int_{A_n(a_\varepsilon,b_\varepsilon)}\vert v_n\vert^{2} dx\nonumber\\
 & \leq C\frac{C^{''}}{c_{a_\varepsilon}^{b_\varepsilon} \Vert u_n\Vert^{2}}\nonumber\\
 & \leq \frac{\varepsilon}{3},\ \ \text{for all}\ n\geq n_0.
  \end{align}
  Putting together \eqref{eq0009}, \eqref{eq00011} and \eqref{eq00013}, we obtain
  $$\int_{\Omega}\frac{f(x,u_n)}{\vert u_n\vert}\vert v_n\vert( v_n^+-v_n^-)dx\leq \varepsilon,\ \ \text{for all}\ n\geq n_0.$$
 It follows, by \eqref{eq0005}, that 
 \begin{equation}\label{eq998}
   \lim\limits_{n\rightarrow +\infty}\frac{\Gamma^{'}(u_n)(u_n^+-u_n^-)}{\Vert u_n\Vert^2}=1.  
 \end{equation}
 On the other side, from \eqref{eq0002} for the nonlocal term, we easily show that
  \begin{equation}\label{eq999}
   \lim\limits_{n\rightarrow +\infty}\frac{\Gamma(u_n)}{\Vert u_n\Vert^2}=0.  
 \end{equation}
Moreover, by Lemma \ref{lemma1}, we have
\begin{align}\label{eq997}
  \left \vert \frac{\Gamma^{'}(u_n)(u_n^+-u_n^-)}{\Vert u_n\Vert^2}\right\vert& \leq  \frac{\Vert \Gamma^{'}(u_n)\Vert_{\mathbb{E}^*}\Vert u_n^+-u_n^-\Vert }{\Vert u_n\Vert^2}\nonumber\\
  & \leq C_3\left\vert \frac{\left(\sqrt{\Gamma^{'}(u_n)u_n}+\Gamma^{'}(u_n)u_n\right)\Vert u_n^+-u_n^-\Vert }{\Vert u_n\Vert^2}\right\vert\nonumber\\
  & \leq C_4\left\vert \frac{\sqrt{\Gamma^{'}(u_n)u_n}+\Gamma^{'}(u_n)u_n}{\Vert u_n\Vert}\right\vert\nonumber\\
  & = C_4\left(\frac{1}{\sqrt{\Vert u_n\Vert}}\sqrt{\frac{4\Gamma(u_n)}{\Vert u_n\Vert}}+\frac{4\Gamma(u_n)}{\Vert u_n\Vert}\right)\longrightarrow 0,\ \ \text{as}\ n\longrightarrow +\infty.
\end{align}
Thus,
$$\lim\limits_{n\rightarrow +\infty}\frac{\Gamma^{'}(u_n)(u_n^+-u_n^-)}{\Vert u_n\Vert^2}=0,$$
which contradicts \eqref{eq998}. Therefore, $\lbrace u_n\rbrace_{n\in\mathbb{N}}$ is bounded in $\mathbb{E}$. This ends the proof.
\end{proof}
Let $\lbrace u_n\rbrace_{n\in\mathbb{N}}\subset\mathbb{E}$ be a $(C)_c$-sequence of the functional $J$, by the previous proposition, the sequence $\lbrace u_n\rbrace_{n\in\mathbb{N}}$ is bounded in $\mathbb{E}$. Since  $\mathbb{E}$ is a reflexive space, up to a subsequence, we may find $u\in \mathbb{E}$ such that
$$u_n\rightharpoonup u,\ \text{in}\ \mathbb{E},$$
$$u_n\rightarrow u,\ \text{in}\ L^{r}_{\text{loc}}\ \text{for all}\ r\in (1,3),$$
and
$$u_n(x)\rightarrow u(x),\ \text{a.a.}\ x\in \mathbb{R}^3.$$
Moreover, it's clear that this $u$ is a critical point of $J$. Setting $w_n:=u_n-u$, then 
$$w_n\rightarrow 0\ \text{in}\ \mathbb{E}.$$
Arguing as in \cite[Lemma 3.7]{Zh} and \cite{D1}, we can prove the following results on the sequence $w_n$.
\begin{lemma}\label{lemma2}
    Under the assumptions of Theorem \ref{thm1}, we have 
    \begin{enumerate}
        \item[$(1)$] $\displaystyle{\lim\limits_{n\rightarrow +\infty}\int_{\mathbb{R}^3}\left(F(x,u_n)-F(x,u)-F(x,w_n)\right)dx=0};$ 
        \item[$(2)$] $\displaystyle{\lim\limits_{n\rightarrow +\infty}\int_{\mathbb{R}^3}\left(f(x,u_n)v-f(x,u)v-f(x,w_n)v\right)dx=0}$, for all $v\in \mathbb{E}$;
        \item[$(3)$] $\displaystyle{\lim\limits_{n\rightarrow +\infty}\left(\Gamma(u_n)-\Gamma(u)-\Gamma(w_n)\right)}$;
        \item[$(4)$] $\displaystyle{\lim\limits_{n\rightarrow +\infty}\left(\Gamma^{'}(u_n)v-\Gamma^{'}(u)v-\Gamma^{'}(w_n)v\right)}$, for all $v\in \mathbb{E}$.
    \end{enumerate}
\end{lemma}
As a consequence,
\begin{lemma}\label{lemma3}
    Under the assumptions of Theorem \ref{thm1}, one has, as $n\longrightarrow +\infty$,
    \begin{enumerate}
        \item[$(1)$] $J(w_n)\longrightarrow c-J(u)$;
        \item[$(2)$] $J^{'}(w_n) \longrightarrow 0$.
    \end{enumerate}
\end{lemma}

Next, we set the set of nontrivial critical points of the functional $J$ as follows
$$\mathcal{J}:=\left\lbrace u\in \mathbb{E}\setminus\lbrace 0\rbrace:\ J^{'}(u)=0\right\rbrace.$$
\begin{prop}\label{prop1}
  Under the assumptions of Theorem \ref{thm1}, the following assertions
hold  
\begin{enumerate}
    \item[$(1)$] $\displaystyle{\vartheta:=\inf\lbrace \Vert u\Vert:\ u\in \mathcal{J}\rbrace>0};$
    \item[$(2)$] $\displaystyle{\theta:=\inf\lbrace J(u):\ u\in \mathcal{J}\rbrace>0}.$
\end{enumerate}
\end{prop}
\begin{proof}
Let's prove $(1)$. Let $u\in \mathcal{J}$, it holds
\begin{align}
    J^{'}(u)(u^+-u^-)=\Vert u\Vert^{2}-\Gamma^{'}(u)(u^+-u^-)-\int_{\mathbb{R}^3}f(x,u)(u^+-u^-)dx=0.
\end{align}
It follows, by \eqref{eq994} and Lemma \ref{lem1}-10, that
$$\Vert u\Vert^2\leq C\Vert u\Vert^4+\varepsilon\Vert u\Vert^2+C_\varepsilon\Vert u\Vert^p$$
for  $p\in (2,2^*)$. Choosing $\varepsilon=\frac{1}{2}$ in the previous inequality, we obtain
$$0<\frac{1}{2}\Vert u\Vert^2\leq C\Vert u\Vert^4+C_\varepsilon\Vert u\Vert^p.$$
Thus, $\Vert u\Vert>0.$\\
$(2)$ Arguing indirectly, suppose that there is a sequence $\lbrace u_n\rbrace_{n\in \mathbb{N}}\subset \mathcal{J}$ such that $J(u_n)\longrightarrow 0$ as $n\longrightarrow +\infty.$
By the first assertion, $\Vert u_n\Vert\geq \vartheta$. Clearly, $u_n$ is a $(C)_0$-sequence
of $J$, and hence is bounded by Proposition \ref{prop2}. Moreover, $u_n$ is nonvanishing. By
the invariance under the translation of $J$, we can assume, up to a translation, that
$$u_n\longrightarrow u\ \in\mathcal{J}.$$
Using Fatou's lemma and Lemma \ref{lemma1}-(8), we infer that 
\begin{align*}
  0&=\lim\limits_{n\rightarrow +\infty} J(u_n)= \lim\limits_{n\rightarrow +\infty} \left(J(u_n)-\frac{1}{2}J^{'}(u_n)u_n\right)\\
  & = \lim\limits_{n\rightarrow +\infty}  \left(\Gamma(u_n)+\int_{\mathbb{R}^3}\widetilde{F}(x,u_n)dx\right)\\
  & \geq \Gamma(u)+\int_{\mathbb{R}^3}\widetilde{F}(x,u)dx\\
  &>0,
\end{align*}
which is a contradiction. This ends the proof.
\end{proof}

Let $[r]$ denote the integer part of $r\in\mathbb{R}$. As a consequence of Lemma \ref{lemma3} and Proposition \ref{prop1}, we
have the following result (see \cite{Rabio,Kr}).
\begin{prop}\label{prop3}
    Assume that assumptions of Theorem \ref{thm1} are fulfill, and let $\lbrace u_n\rbrace_{n\in \mathbb{N}}\subset \mathbb{E}$ be a $(C)_c$-sequence of $J$. Then, one of the following assertions holds
    \begin{enumerate}
        \item[$(1)$] $u_n\longrightarrow 0$, and hence $c=0$.
        \item[$(2)$] $c\geq \theta$ and there exist a positive integer $\ell\leq [\frac{c}{\theta}]$, points $\overline{u}_1,\cdots,\overline{u}_\ell\in \mathcal{J}$, a
subsequence denoted again by $u_n$, and sequences $\lbrace a_n^k\rbrace_{n\in \mathbb{N}}\subset\mathbb{Z}^3$, $k=1,\cdots,\ell$ such that
$$
\left\lbrace
\begin{array}{l}
 \displaystyle{\left\Vert u_n-\sum_{k=1}^{\ell} a_n^k*\overline{u}_k\right\Vert\longrightarrow 0\ \text{as}\ n\longrightarrow +\infty},\\
 \ \\
\displaystyle{\vert a_n^k-a_n^m\vert\longrightarrow +\infty, \ \text{for all}\ k\neq m\ \text{as}\ n\longrightarrow 
 +\infty},\\
  \ \\
 \text{and}\\
  \ \\
\displaystyle{ \sum_{k=1}^{\ell}J(\overline{u}_k)=0.}
\end{array}
\right.
$$
    \end{enumerate}
\end{prop}
\section{Proof of our main result}
As usual, in variational problems, to prove the existence of weak solutions it is  enough to find critical points of the energy functional associated with the problem. To this end, we shall use the following abstract theorem which
is taken from \cite{D1}.\\

First of all, we state definitions and notations. Let $\mathbb{W}$ be a Banach space with direct sum $\mathbb{W}=\mathbb{X}\oplus\mathbb{Y}$ and corresponding projections
$P_{\mathbb{X}}, P_{\mathbb{Y}}$ onto $\mathbb{X}, \mathbb{Y}$. Let $\mathcal{S}\subset  \mathbb{X}^*$ be a dense subset, for each $s\in \mathcal{S}$ there
is a semi-norm on $\mathbb{W}$ defined by
$$p_s:\mathbb{W}\longrightarrow \mathbb{R},\ p_s(u):=\vert s(x)\vert+\Vert y\Vert,\ \ \text{for all}\ u=x+y\in\mathbb{W}.$$
We denote by $\mathcal{T}_{\mathcal{S}}$ the topology induced by semi-norm family $\lbrace p_s\rbrace_{s\in \mathcal{S}}$, $w^*$ denote the
$\text{weak}^*$-topology on $\mathbb{W}^*$.
For a functional $\Phi\in C^{1}(\mathbb{W},\mathbb{R})$ we write 
$$\Phi_c:=\lbrace u\in \mathbb{W}:\ \ \Phi(u)\geq c \rbrace.$$

We may suppose that 
\begin{enumerate}
    \item[$(\Phi_1)$] for any $c\in \mathbb{R}$, super level $\Phi_c$ is $\mathcal{T}_{\mathcal{S}}$-closed and $\Phi^{'}:(\Phi_c,\mathcal{T}_{\mathcal{S}})\longrightarrow (\mathbb{E}^*,w^*)$ is continuous;
    \item[$(\Phi_2)$] for any $c>0$, there exists $\xi>0$ such that $\Vert u\Vert<\xi \Vert P_{\mathbb{Y}}u\Vert$, for all $u\in \Phi_c$;
    \item[$(\Phi_3)$] there exists $r>0$ such that $\rho:=\inf \Phi(S_r\cap \mathbb{Y})>0$, where $S_r:=\lbrace u\in \mathbb{W}:\ \ \Vert u\Vert=r\rbrace$.
\end{enumerate}
\begin{thm}\label{thm3}
  Let the assumptions $(\Phi_1)-(\Phi_3)$ be satisfied and suppose that there exist $R>r>0$ and $e\in \mathbb{Y}$ with $\Vert e\Vert=1$ such that $$\sup \Phi(\partial Q)\leq \rho,$$
  where $Q:=\lbrace u=x+te:\ \ x\in \mathbb{X}, t\geq 0,\ \Vert u\Vert\leq R\rbrace$. Then, $\Phi$ has a $(C)_c$-sequence with
  $$\rho \leq c\leq \sup\Phi(Q).$$
\end{thm}
\begin{proof}[\textbf{Proof of Theorem \ref{thm1}}]
  Taking $\mathbb{W}=\mathbb{E}$, $\mathbb{X}=\mathbb{E}^-$, $\mathbb{Y}=\mathbb{E}^+$, and $\Phi=J$ in the previous theorem. By Lemma \ref{lemma8} and Proposition \ref{prop7}, we
see that $(\Phi_1)-(\Phi_3)$ are satisfied. The Proposition \ref{prop8} shows that $J$ possesses the linking structure of Theorem \ref{thm3}. Therefore, there exists $\lbrace u_n\rbrace_{n\in \mathbb{N}}\subset \mathbb{E}$ a $(C)_c$-sequence of $J$ at level $c$. By Proposition \ref{prop2}, $u_n$ is bounded in $\mathbb{E}$. Let
$$\delta:=\limsup_{n\rightarrow +\infty}\sup_{y\in \mathbb{R}^3}\int_{B(y,1)}\vert u_n\vert^2dx.$$
Here, we distinguish two possibilities for $\delta$:  $\delta=0$ or $\delta>0$.\\
If $\delta=0$, by Lion's concentration compactness principle in \cite[Lemma 1.21]{th2}, we have that
$$u_n\longrightarrow 0\ \text{in}\ L^{r},\ \text{for all}\ r\in(2,3).$$
It follows, from \eqref{eq994} and Lemma \ref{lem1}-(10), that
$$\left.
\begin{array}{c}
\displaystyle{\int_{\mathbb{R}^3}f(x,u_n)u_ndx\longrightarrow 0},\\
\ \\
\displaystyle{\int_{\mathbb{R}^3}F(x,u_n)dx\longrightarrow 0},\\
\ \\
 \text{and}\\
\ \\
\displaystyle{\Gamma (u_n)\longrightarrow 0}
\end{array}
\right\rbrace\ \text{as}\ n\longrightarrow +\infty.
$$
Consequently,
\begin{align*}
  c&=\lim\limits_{n\rightarrow +\infty} J(u_n)= \lim\limits_{n\rightarrow +\infty} \left(J(u_n)-\frac{1}{2}J^{'}(u_n)u_n\right)\\
  & = \lim\limits_{n\rightarrow +\infty}  \left(\Gamma(u_n)+\int_{\mathbb{R}^3}\widetilde{F}(x,u_n)dx\right)\\
  &=0,
\end{align*}
this is a contradiction. Thus, $\delta >0$.\\
Going if necessary to a subsequence, we may assume the existence of $k_n\subset \mathbb{Z}^3$ such that
$$\int_{B(k_n,1+\sqrt{3})}\vert u_n\vert^2dx>\frac{\delta}{2}.$$
Lets define $v_n(x) := u_n(x + k_n)$. So, 
\begin{equation}\label{eq993}
\int_{B(0,1+\sqrt{3})}\vert v_n\vert^2dx>\frac{\delta}{2}.
\end{equation}
Since $J$ and $J^{'}$ are $\mathbb{Z}^3$-translation invariant, we obtain $\Vert v_n\Vert=\Vert u_n\Vert$ and 
\begin{equation}\label{eq992}
J(v_n)\longrightarrow c\ \ \text{and}\ \ (1+\Vert v_n\Vert) J^{'}(v_n)\longrightarrow 0,\ \ \text{as}\ \ n\longrightarrow +\infty.
\end{equation}
Passing to a subsequence, we have $v_n\rightharpoonup v$ in $\mathbb{E}$, $v_n\rightarrow v$ in $L^r$, for all $r\in [1, 3)$ and $v_n(x)\rightarrow v(x)$ a.e. on $\mathbb{R}^3$. Hence, it follows, from \eqref{eq993}  and \eqref{eq992}, that $J^{'}(v)=0$ and
$v\not\equiv 0$. This shows that $v\in \mathcal{J}$. Therefore, $v\in \mathbb{E}$ is a nontrivial weak solution of problem \eqref{P}. Thus,  $(v,\phi_v)$ is a pair of solutions for system \eqref{DBP}.
\end{proof}

\end{document}